\documentclass[11pt]{amsart}
\usepackage{amsfonts,amssymb,amsmath,amsthm}
\usepackage{url}
\usepackage{enumerate}
\usepackage[all]{xy}
\usepackage[utf8]{inputenc}
\usepackage{amsmath,amsfonts,amsthm,url,color,amssymb}

\usepackage{graphicx}
\usepackage{algpseudocode, algorithm}
\usepackage{hyperref}
\usepackage{todonotes}

\usepackage[norefs,nocites]{refcheck}

\newcommand{\Z}{\mathbb Z}

\newcommand{\Tr}{\mathrm{Tr}}
\newcommand{\fqn}{\mathbb{F}_{q^n}}
\newcommand{\F}{\mathbb{F}}

\newtheorem{theorem}{Theorem}[section]
\newtheorem{lemma}[theorem]{Lemma}

\newtheorem{corollary}[theorem]{Corollary}
\newtheorem{proposition}[theorem]{Proposition}
\theoremstyle{definition}
\newtheorem{definition}[theorem]{Definition}

\theoremstyle{remark}
\usepackage{amssymb}
\newtheorem{remark}[theorem]{Remark}
\usepackage{enumerate}
\author[L. Reis]{Lucas Reis}
\address{Departamento de Matem\'{a}tica, Universidade Federal de Minas Gerais, Belo Horizonte MG, 31270901, Brazil}
\email{lucasreismat@gmail.com}

\begin{document}

\title[Paley-like graphs from vector spaces]{Paley-like graphs over finite fields from vector spaces} 

\begin{abstract}
Motivated by the well-known Paley graphs over finite fields and their generalizations, in this paper we explore a natural multiplicative-additive analogue of such graphs arising from vector spaces over finite fields. Namely, if $n\ge 2$ and $U\subsetneq \F_{q^n}$ is an $\F_q$-vector space, $G_{U}$ is the (undirected) graph with vertex set $V(G_U)=\F_{q^n}$ and edge set $E(G_U)=\{(a, b)\in \F_{q^n}^2\,|\, a\ne b, ab\in U\}$. We describe the structure of an arbitrary maximal clique in $G_U$ and provide  bounds on the clique number $\omega(G_U)$ of $G_U$. In particular, we compute the largest possible value of $\omega(G_U)$ for arbitrary $q$ and $n$. Moreover, we obtain the exact value of $\omega(G_U)$ when $U\subsetneq \F_{q^n}$ is any $\F_q$-vector space of dimension $d_U\in \{1, 2, n-1\}$. %We also propose an interesting open problem.

\end{abstract}

\maketitle

\section{Introduction}
Let $q$ be a prime power and let $\F_q$ be the finite field with $q$ elements. When $q$ is odd, the Paley graph over $\F_q$ is the graph $\mathcal G_q$ with vertex set $V(\mathcal G_q)=\F_q$ and edge set $E(\mathcal G_q)\subseteq \F_q^2$ such that $(a, b)\in E(\mathcal G_q)$ if and only if $a-b$ is a non zero square of $\F_q$. To obtain an undirected graph we further require that $-1$ is a square in $\F_q$, i.e., $q\equiv 1\pmod 4$. A main problem in this setting is to find the clique number of $\mathcal G_q$, i.e., the greatest positive integer $N$ such that $\mathcal G_q$ contains a clique of size $N$. The Paley graphs were further generalized imposing the condition that $a-b$ is a non zero $d$-th power, where $q\equiv 1\pmod {2d}$ (see~\cite{cohen}). With the aim of studying Waring's problem over finite fields, in~\cite{pod} the authors consider another generalization of Paley graphs removing the condition $q\equiv 1\pmod {2d}$. The clique number for these kind of graphs was also studied in ~\cite{csaba} and~\cite{yip1} (see also~\cite{yip2}). We refer to~\cite{paley} for more details on such graphs, including properties and connections to other research topics, and~\cite{yip} for a nice overview on bounds for their clique numbers.

The Paley graph and their generalizations explore the interaction between differences in $\F_q$ and subgroups of the multiplicative group $\F_q^*$. A natural multiplicative-additive analogue is to consider how quotients $ab^{-1}$ interact with subgroups of the additive group $(\F_q, +)$, i.e., vector subspaces $U\subseteq \F_q$. We may naturally construct graphs in this analogy but in order to still obtain an undirected graph we should impose the condition that $u^{-1}\in U$ for every non zero $u\in U$, which is quite restrictive. In fact, according to Theorem 3 in~\cite{bence}, the latter implies that $U$ is either a subfield of $\F_q$ or $U=\{x\in \F_q\,|\, x^{p^d}+x=0\}$, where $q=p^{m}$, $p\ne 2$ is a prime and $m\equiv 0\pmod {2d}$.

In light of the previous observations, in this paper we consider the following setting. For an integer $n\ge 2$, let $\F_{q^n}$ be the unique $n$-degree extension of $\F_{q}$. Then $\F_{q^n}$ can be regarded as an $\F_q$-vector space of dimension $n$. If $U\subsetneq \F_{q^n}$ is an $\F_q$-vector space, we set $G_{U}$ as the (undirected) graph with vertex set $V(G_U)=\F_{q^n}$ and edge set $E(G_U)=\{(a, b)\in \F_{q^n}^2\,|\, a\ne b, ab\in U\}$. Here we are mainly interested in the clique number $\omega(G_U)$ of $G_U$. We combine tools from linear algebra and additive combinatorics to estimate the number $\omega(G_U)$.

Our main results, Theorems~\ref{thm:main} and~\ref{thm:main3}, provide non trivial bounds on $\omega(G_U)$ and are presented in Section 2. The rest of the paper goes as follows. In Section 3 we present Theorem~\ref{thm:main1} which describes the structure of a maximal clique in $G_U$ (i.e. a clique that cannot be extended to a larger clique). As an application of the latter, in  Proposition~\ref{prop:basic} we compute $\omega(G_U)$ when $U$ has dimension at most two, and we also prove Theorem~\ref{thm:main3}. In Section 4 we discuss some issues on bilinear forms over finite fields and provide auxiliary results. Using the machinery developed in Section 4, in Section 5 we completely describe the value of $\omega(G_U)$ when $U$ has dimension $n-1$.% Finally, in Section 6, we provide conclusions and open problems.

\section{Main results}
Throughout this paper, $q$ is a prime power, $n\ge 2$ is an integer and $\F_{q^n}$ is the finite field with $q^n$ elements. Moreover, $\F_{q^n}^*=\F_{q^n}\setminus \{0\}$ and an element $a\in \F_{q^n}$ is a square of $\F_{q^d}$ if $a=b^2$ for some $b\in \F_{q^d}$. Otherwise, $a$ is a non square. When $d=n$, we simply say that $a$ is a square or a non square. Before we state our main results, we introduce two useful definitions.

\begin{definition}For an $\F_q$-vector space $U\subsetneq \F_{q^n}$, $G_{U}$ is the graph with vertex set $V(G_U)=\F_{q^n}$ and edge set $E(G_U)=\{(a, b)\in \F_{q^n}^2\,|\, a\ne b, ab\in U\}$. We set $d_U$ the dimension of $U$ over $\F_q$ and $\omega(G_U)$ the clique number of $G_U$.
\end{definition}

Throughout this paper, we do not take any distinction between an element of $\F_{q^n}$ and its associated vertex in $G_U$.

\begin{definition}
For sets $S, T\subseteq \F_{q^n}$, $S\pm T:=\{s\pm t\,|\, s\in S, t\in T\}$ and $S\cdot T:=\{st\,|\, s\in S, t\in T\}$. If $S=\{a\}$ is a singleton we simply write $a\pm T$ and $aT$. Moreover, $\#S$ denotes the cardinality of $S$.
\end{definition}

The following result provides a uniform upper bound on the numbers $\omega(G_U)$ for $n\ge 2$.

\begin{theorem}\label{thm:main}
Let $q$ be a prime power and $n\ge 2$. If $\omega_{q, n}$ denotes the largest value of $\omega(G_U)$ as $U$ runs over all the proper $\F_q$-vector spaces in $\F_{q^n}$, then the following hold:
\begin{enumerate}[1.]
\item $\omega_{2, n}=n+1$ for $n=2, 3, 4, 5$; 
\item If $q> 2$ or $n\ge 6$, $\omega_{q, n}=q^{(n-1)/2}+1$ if $n$ is odd and $\omega_{q, n}=q^{n/2}$ if $n$ is even.
\end{enumerate}
\end{theorem}

We observe that if $U\subseteq U'\subsetneq \F_{q^n}$ are two $\F_q$-vector spaces, then $G_U$ is a subgraph of $G_{U'}$ and so $\omega(G_U)\le \omega(G_{U'})$. Since any proper $\F_q$-vector space of $\F_{q^n}$ is contained in an $\F_q$-vector space of dimension $n-1$, it follows that $\omega_{q, n}=\omega(G_U)$ for some $U$ of dimension $n-1$. In particular, Theorem~\ref{thm:main} is a straightforward consequence of Theorem~\ref{thm:n-1}, where the exact value of $\omega(G_U)$ is computed for arbitrary $U$ of dimension $n-1$. 

%The main idea behind the proof of Theorem~\ref{thm:n-1} is to describe the relation $(a, b)\in E(G_U)$ by means of the condition $B(a, b)=0$, where $B$ is a non degenerate symmetric $\F_q$-bilinear form of $\F_{q^n}$, and essentially employ an extended version of a result from~\cite{ah}. This interaction with bilinear forms over finite fields is done by showing that any $\F_q$-vector space $U\subsetneq \F_{q^n}$ of dimension $n-1$ comprise the solutions to the equation $T(\delta x)=0$ for some $\delta\in \F_{q^n}^*$, where $T:\F_{q^n}\to \F_q$ is the trace function of $\F_{q^n}$ over $\F_q$.

Theorem~\ref{thm:main} readily implies the bound $\omega(G_U)\le \omega_{q, n}$ for arbitrary $U$. If $d_U<\lfloor n/2\rfloor$ or $q=2$ and $(n, d_U)\in \{(3, 1),  (5, 2)\}$, this bound is worse than the trivial bound $\omega(G_U)\le q^{d_U}+1$. Moreover, the existence of subfields $\F_q\subset \mathbb F_{q^d}\subset \F_{q^n}$ yields obstruction to any significant improvement on the trivial bound: if $a^2\F_{q^d}\subseteq U$ with $a\ne 0$, the vertices from $a\F_{q^d}$ yield a clique in $G_U$ and then $\omega(G_U)\ge q^{d}$. In this context, we obtain the following result.

\begin{theorem}\label{thm:main3}
Let $q$ be a prime power, $n\ge 2$  and let $U\subsetneq \F_{q^n}$ be an $\F_q$-vector space of dimension $d_U\ge 1$. If $U$ does not contain a non zero square, then $\omega(G_U)\le d_U+2$. Assume now that $\omega(G_U)>d_U+2$ and let $1\le D_U\le d_U$ be the greatest divisor of $n$ such that $a^2\F_{q^{D_U}}\subseteq U$ for some $a\ne 0$. Then $\omega(G_U)=q^t+r$  for some $t, r\ge 0$ such that $r+t\le d_U$ and 
$$1\le t\le \kappa_U:=\max\left\{D_U, \frac{7}{8}d_U+\frac{7}{32\log_2q}\right\}.$$
In particular, $q^{D_U}\le \omega(G_U)\le q^{\kappa_U}+d_U$. 
\end{theorem}

The proof of Theorem~\ref{thm:main3} follows by a tricky combination  of Theorems~\ref{thm:main} and ~\ref{thm:main1}, and a beautiful result from Bourgain and Glibichuk on sum-product estimates over finite fields: see Subsection 3.1 for more details. As follows, we provide two interesting applications of Theorem~\ref{thm:main3}.

We observe that if $\omega(G_U)=q^t+r>d_U+2$ and $D_U$ are as in Theorem~\ref{thm:main} then $D_U\le n_U$, where $n_U$ is the largest divisor of $n$ such that $n_U\le d_U$. In particular, we obtain the following corollary.

\begin{corollary}
Let $q$ be a prime power, $n\ge 2$  and let $U\subsetneq \F_{q^n}$ be an $\F_q$-vector space of dimension $d_U\ge 1$ such that $\omega(G_U)=q^t+r>d_U+2$ with $t>0$ and $r\le d_U$. If $1\le n_U<n$ is the largest divisor of $n$ such that $n_U\le d_U$, then $t\le \max\left\{n_U, \frac{7}{8}d_U+\frac{7}{32\log_2q}\right\}$. In particular, the following hold:

\begin{enumerate}[1.]
\item $n_U=1$ if $n$ is a prime;
\item If  $1\le m\le 7$ and $d_U<\frac{n}{m}$, then $n_U\le \frac{n}{m+1}$ and so $t<\frac{7n}{8m}+\frac{7}{32\log_2q}$.
\end{enumerate}
\end{corollary}

It is direct to verify that  $f(s)=q^s+d_U-s$ is a strictly increasing function on $s\in \Z_{\ge 1}$ and $d_U+2\le q^{d_U}$ for arbitrary $q, d_U\ge 2$. In particular, Theorem~\ref{thm:main3} implies the upper bound $\omega(G_U)\le q^{d_U}$ if $d_U\ge 2$. The following corollary describes the $\F_q$-vector spaces $U$ reaching this upper bound.

\begin{corollary}\label{thm:main2}
For any $\F_q$-vector space $U\subsetneq \F_{q^n}$ of dimension $d_U\ge 2$, we have that $\omega(G_U)\le q^{d_U}$ with equality if and only if one of the following holds:

\begin{enumerate}[1.]
\item $q=d_U=2$. In this case, $U=\{0, u, v, u+v\}$ for non zero $u\ne v$ and any maximal clique in $G_U$ is generated by a set of vertices $\{0, a, \frac{u}{a}, \frac{v}{a}\}$ with $a^2=\frac{uv}{u+v}$ (since the field is of characteristic two, such an $a\in \F_{q^{n}}$ always exists);
\item $U=a^2\F_{q^{d_U}}$, where $d_U<n$ divides $n$ and $a\ne 0$. In this case, a maximal clique in $G_U$ is generated by the set $a \F_{q^{d_U}}$. 
\end{enumerate}
\end{corollary}

\begin{proof}
The case $q=d_U=2$ is directly verified so we assume that $(q, d_U)\ne (2, 2)$. In particular,  $d_U+2< q^{d_U}$ and then, from Theorem~\ref{thm:main3}, it suffices to consider the case where $U$ contains a non zero square and $\omega(G_U)=q^t+r$ for some $t, r\ge 0$ such that $t\ge 1$ and $r+t\le d_U$. Therefore, $\omega(G_U)\le q^t+d_U-t$ and we have seen that $f(s)=q^s+d_U-s$ is a strictly increasing function on $s\in \Z_{\ge 1}$ for arbitrary $q\ge 2$. In particular, if $$k:=\left\lfloor \max\left\{D_U, \frac{7}{8}d_U+\frac{7}{32\log_2q}\right\}\right\rfloor,$$ then $1\le k \le  d_U$  and $k=d_U$ if and only if $d_U=D_U$. Theorem~\ref{thm:main3} entails that $t\le k$ and so $\omega(G_U)\le f(t)\le f(k)\le f(d_U)=q^{d_U}$ with equality $f(t)=f(d_U)$ if and only if $t=k=d_U$. Therefore, $\omega(G_U)=q^{d_U}$ if and only if $d_U=D_U$. In other words, $U$ must contain a set of the form $a^2\F_{q^{d_U}}, a\ne 0$. Since $U$ has dimension $d_U$, the latter is equivalent to $U=a^2\F_{q^{d_U}}$ for some $a\ne 0$. Since $U\subsetneq \F_{q^n}$, it is clear that $d_U<n$ must be a divisor of $n$.
\end{proof}

\section{Miscellanea}
The following theorem provides a special decomposition of the vertices of a maximal clique in $G_U$.
% and it is a crucial tool in the proof of our main results.

\begin{theorem}\label{thm:main1}
Let $U\subsetneq \F_{q^n}$ be an $\F_q$-vector space of dimension $d_U\ge 1$ and let $G$ be any maximal clique in $G_U$. If $V(G)_2$ denotes the set of elements $\alpha\in V(G)$ such that $\alpha^2\in U$, then the following hold:
\begin{enumerate}[1.]
\item $V(G)_2$ is an $\F_q$-vector space; 
\item The elements of $V(G)_1:=V(G)\setminus V(G)_2$ are $\F_q$-linearly independent;
\item If $V(G)_1$ is non empty and $W$ is the $\F_q$-vector space generated by the elements of $V(G)_1$, then $W\cap V(G)_2=\{0\}$.
\end{enumerate}
In particular, $\omega(G_U)=q^t+r$ for some $t, r\ge 0$ such that either $t=0$ and $r\le d_U+1$ or $t\ge 1$ and $r+t\le d_U$.
\end{theorem}

\begin{proof}
We prove the items separately.
\begin{enumerate}[1.]
\item Since $G$ is maximal, $0\in V(G)_2$ and it suffices to show that, for any $u, v\in V(G)_2, w\in V(G)$ and $\lambda\in \F_{q}$, we have that $(u+\lambda v)^2, (u+\lambda v)w\in U$. From hypothesis, $u^2, v^2, uv\in U$ and then $(u+\lambda v)^2=u^2+2\lambda uv+\lambda^2v^2\in U$. Moreover, since $u, v\in V(G)_2$, we have that $uw, vw\in U$ for every $w\in V(G)$. Hence $(u+\lambda v)w\in U$ for every $w\in V(G)$.

\item The result is trivial if $V(G)_1$ has at most one element, so we assume otherwise. Suppose that there exists an integer $s\ge 2$ and elements $\beta_1, \ldots, \beta_s\in V(G)_1, c_1, \ldots, c_s\in \F_q$ such that $\sum_{i=1}^sc_i\beta_i=0$.
Hence, for every $1\le k\le s$, we have that
$$\beta_k\cdot \sum_{i=1}^sc_i\beta_i=0\in U.$$
From hypothesis, $\beta_k\beta_j\in U$ for every $j\ne k$ and so we conclude that $c_k\beta_k^2\in U$. Since $\beta_k^2\not\in U$, we have that $c_k=0$.  Since $k$ is arbitrary, it follows that $c_1=\cdots=c_s=0$.
\item It suffices to consider the case where $V(G)_1$ is non empty and $V(G)_2$ has dimension at least one. Let $\{\alpha_1, \ldots, \alpha_m\}$ be an $\F_q$-basis for $V(G)_2$. Suppose that there exist elements $\beta_1, \ldots, \beta_s\in V(G)_1$, $c_1, \ldots, c_s\in \F_q$ and $e_1, \ldots, e_m\in \F_q$ such that $\sum_{i=1}^{m}e_i\alpha_i+\sum_{j=1}^sc_i\beta_j=0$. Hence, for every $1\le k\le s$, we have that 
$$\beta_k\left(\sum_{j=1}^{m}e_j\alpha_j+\sum_{i=1}^sc_i\beta_i\right)=0.$$
From hypothesis, $\beta_k\alpha_i, \beta_k\beta_j\in U$ for every $1\le i\le m$ and every $1\le j\le s$ with $j\ne k$. In particular, we conclude that $c_k\beta_k^2\in U$. Since $\beta_k^2\not\in U$, we have that $c_k=0$. Since $1\le k\le s$ is arbitrary, we conclude that $\sum_{j=1}^{m}e_i\alpha_i=0$ and so $e_j=0$ for every $1\le j\le m$ since the $\alpha_j$'s are $\F_q$-linearly independent.
\end{enumerate}

To conclude the proof, let $H$ be a maximal clique in $G_U$ with $\#V(H)=\omega(G_U)$. If $q^t$ and $r$ are the cardinality of the sets $V(H)_2$ and $V(H)_1$, respectively, we have that $\omega(G_U)=q^t+r$. If $r>0$ and $\beta \in V(H)_1$, we have that $\beta\ne 0$ and item 2 entails that the $r-1$ elements $\beta\cdot w, w\in V(H)_1\setminus \{\beta\}$ are $\F_q$-linearly independent. However, from hypothesis, such elements also lie in $U$. Hence $r-1\le d_U$ and so $r\le d_U+1$. Moreover if $t>0$, there exists a non zero element $\alpha \in V(H)_2$. In particular, $\alpha (W+ V(G)_2)\subseteq U$, where $W$ is the $\F_q$-vector space generated by the elements of $V(G)_1$. From item 3, the $\F_q$-vector space $W+ V(G)_2$ has dimension $r+t$ and, since $\alpha\ne 0$, we conclude that $r+t\le d_U$.
\end{proof}

We obtain the following bounds for $\omega(G_U)$.

\begin{corollary}\label{cor:1}
For any $\F_q$-subspace $U\subsetneq \F_{q^n}$ of dimension $d_U\ge 2$, either $\omega(G_U)=q^{d_U}$ or $\omega(G_U)\le q^{d_U-1}+1$.
\end{corollary}
\begin{proof}
If $q=d_U=2$, Corollary~\ref{thm:main2} implies that $\omega(G_U)=q^{d_U}$ so the statement is true in this case.  Suppose that $(q, d_U)\ne (2, 2)$. From Theorem~\ref{thm:main1}, either $\omega(G_U)\le d_U+2$ or $\omega(G_U)\le q^s+d_U-s$ for some $1\le s\le d_U$. It is direct to verify that $f(s)=q^s+d_U-s$ is a strictly increasing function on $s\in \Z_{\ge 1}$. The result now follows by the fact that $d_U+2\le q^{d_U-1}+1$ for every $q, d_U\ge 2$ with $(q, d_U)\ne (2, 2)$.
\end{proof}

\begin{lemma}\label{cor:2}
If $U\subsetneq \F_{q^n}$ is an $\F_q$-vector space of dimension $d_U\ge 1$, then $\omega(G_U)\ge 3$. Moreover, if $U$ contains a non zero square, we have that $\omega(G_U)\ge q+\min\{1, d_U-1\}$ and, otherwise, $\omega(G_U)=3$.
\end{lemma}

\begin{proof}
Since $\#U\le q^{n-1}<\frac{q^n-1}{\gcd(q-1, 2)}=\#\{a^2\,|\, a\in \F_{q^n}^*\}$ for every $q, n\ge 2$, there exists $a\in \F_{q^n}^*$ such that $a^2\in \F_{q^n}^* \setminus U$ and, for such an $a$, $\{0, a, \frac{u}{a}\}$ yields a clique of size $3$ in $G_U$. In particular, $\omega(G_U)\ge 3$ for arbitrary $U$.
Suppose now that $a^2\in U$ for some $a\ne 0$. In this case, any two vertices from $a \F_q$ are connected in $G_U$ and they are also all connected to another vertex $\frac{u}{a}\not\in a\F_q$ if $d_U>1$ (just pick any element $u\in U\setminus a^2\F_q$). 

It remains to prove that $\omega(G_U)<4$ if $U$ does not contain a non zero square. Since any maximal clique of $G_U$ contains the element $0$, if the former does not hold, we conclude that $\{0, a, \frac{u}{a}, \frac{v}{a}\}$ is a clique in $G_U$ for some $a\ne 0$ and some distinct $u, v\in  U$ (observe that $v/a, u/a\ne a$ since $U$ does not contain a non zero square). From hypothesis, $u, v$ are non squares and then $uv$ is a square. In particular, $\frac{uv}{a^2}$ is a non zero square, a contradiction with $\frac{uv}{a^2}\in U$. 
\end{proof}

As a nice application of the previous results, we completely describe $\omega(G_U)$ when $d_U=1, 2$. 
\begin{proposition}\label{prop:basic}
Let $U\subsetneq \F_{q^n}$ be an $\F_q$-vector space of dimension $d_U$. Then the following hold:

\begin{enumerate}[1.]
\item Assume $d_U=1$. Then $\omega(G_U)=3$ if $q=2, 3$ and $\omega(G_U)\in \{3, q\}$ if $q>3$. Moreover, in the latter case, $\omega(G_U)=q$ or $3$ according to whether $U$ contains a non zero square or not, respectively.

\item Assume $d_U=2$. Then $\omega(G_U)\in \{3, q+1, q^2\}$ and the following hold:
\begin{itemize}
\item $\omega(G_U)=3$ if and only if $U$ does not contain a non zero square;
\item $\omega(G_U)=q^2$ if and only if $q=2$ or $U=a^2\F_{q^2}$ with $a\ne 0$.
\end{itemize}
\end{enumerate}
\end{proposition}

\begin{proof}
We prove items 1 and 2 separately.
\begin{enumerate}[1.]
\item This follows directly by Theorem~\ref{thm:main1} and Lemma~\ref{cor:2}.
\item Corollary~\ref{thm:main2} readily proves the result for $q=2$. Suppose that $q>2$. Theorem~\ref{thm:main1} entails that $\omega(G_U)\in \{3, q+1, q^2\}$. Since $q+\min\{1, d_U-1\}=q+1>3$, Lemma~\ref{cor:2} implies that $\omega(G_U)=3$ if and only if $U$ does not contain a non zero square. Moreover, for $q>2$, Corollary~\ref{thm:main2} entails that $\omega(G_U)=q^2=q^{d_U}$ if and only if $U=a^2\F_{q^2}$ for some $a\ne 0$. 
\end{enumerate}
\end{proof}

\subsection{Proof of Theorem~\ref{thm:main3}}
We have the following result (see Lemma 5 in ~\cite{bourgain}).

\begin{lemma}\label{lem:add}
For arbitrary sets $A, B\subseteq \F_{q^n}$, such that $\#A > 1$ and $B$ is not contained in a proper subfield of $\F_{q^n}$,
$$\max\{\# (A+A\cdot B) , \# (A-A\cdot B)\}\ge 2^{-1/4}\cdot \min \{\# A(\#B)^{1/7}, (\# A)^{6/7}q^{n/7}\}.$$

\end{lemma}

We proceed to the proof of Theorem~\ref{thm:main3}. Let $H$ be a maximal clique in $G_U$ with $\# V(H)=\omega(G_U)$ and let $V(H)_2$ with $\# V(H)_2=q^t$ be as in Theorem~\ref{thm:main1}. In particular, $\omega(G)=q^t+r$ for some $t, r\ge 0$ such that either $t=0$ and $r\le d_U+1$ or $t>0$ and $r+t\le d_U$. If $t=0$, then $\omega(G_U)\le d_U+2$ and so the case $t=0$ is settled. From now and on we assume that $\omega(G_U)>d_U+2$, hence $t>0$ and $U$ contains a non zero square. In particular, the case $t\le D_U$ is also settled and we may assume that $t>D_U$.  In this case, $t\ge D_U+1\ge 2$ and there exists a non zero element $\alpha \in V(H)_2$. Hence $\alpha^2\in U$ and so $(a_1, a_2)\in E(G_U)$ if and only if $(\alpha ^{-1}a_1, \alpha^{-1}a_2)\in E(G_{\alpha^{-2}U})$. The latter implies that the graphs $G_U$ and $G_{\alpha^{-2}U}$ are isomorphic and, in particular,  we may assume that $1\in V(H)_2$.  

Since $U$ is an $\F_q$-vector space, $1\in V(H)_2$ and $V(H)_2\cdot V(H)_2\subseteq U$, we conclude that the $\F_q$-vector space $\mathcal V$ generated by the elements of the sets $V(H)_2\pm V(H)_2\cdot V(H)_2$ is contained in $U$. Moreover, since $V(H)_2$ is an $\F_q$-vector space, any proper subfield of $\F_{q^n}$ containing $V(H)_2$ necessarily contains $\F_q$. Let $s\le n$ be the smallest positive integer such that $V(H)_2$ is contained in $\F_{q^s}$, hence $V(H_2)$ is not contained in any proper subfield of $\F_{q^s}$ and $t\le s$. In this case, $\mathcal V\subseteq \F_{q^s}$ and then $\mathcal V\subseteq \F_{q^s}\cap U$. Since $s\ge t>D_U$, the maximality of $D_U$ implies that $\F_{q^s}\cap U\ne \F_{q^s}$ and so $t, d_{\mathcal V}\le s-1$. In particular, $s\ge t+1\ge 3$. 

We observe that $V(H)_2$ yields a clique in $G_{\mathcal V}$ and, in particular, $\omega(G_{\mathcal V})\ge q^t$. Since $s\ge 3$ and $\mathcal V\subsetneq \F_{q^s}$, Theorem~\ref{thm:main} implies that $q^t\le \omega_{q, s}$ and we easily verify that $\omega_{q, s}<q^{(s+1)/2}$ for any $s\ge 3$ and $q\ge 2$ with $(q, s)\ne (2, 3)$. However, if $(q, s)=(2, 3)$, recall that $2\le t\le s-1$ and so we necessarily have that $t=s-1=2$. In this case, $$\omega(G_U)=q^t+r= 4+r=(r+t)+2\le d_U+2,$$ a contradiction with our assumption $\omega(G_U)>d_U+2$. In particular, we have proved that $t\le s/2$ and then, for $A=B=V(H)_2$, we have that 
$$\# A(\#B)^{1/7}=q^{\frac{8}{7}t}\le q^{\frac{6t+s}{7}}=(\# A)^{6/7}q^{s/7}.$$ 
Since $B=V(H)_2$ is not contained in any subfield of $\F_{q^s}$ and $A\pm A\cdot B \subseteq U$, Lemma~\ref{lem:add} entails that
$$q^{d_U}=\# U\ge \max\{\# (A+A\cdot B) , \# (A-A\cdot B)\}\ge 2^{-1/4} q^{\frac{8}{7}t},$$
from where the result follows.

\section{Bilinear forms over finite fields}
In this section we provide some basics on bilinear and quadratic forms over finite fields and also two auxiliary results that are further used in Section 5. Let $B$ be a non degenerate symmetric $\F_q$-bilinear form of $\F_{q^n}$ (regarded as an $\F_q$-vector space of dimension $n$) and let $Q_B(w)=B(w, w)$ be its associated quadratic form. Moreover, if $q$ is odd, let $\chi_q$ denote the quadratic character of $\F_q$, i.e., for $a\in \F_q^*$, $\chi_q(a)=1$ or $\chi_q(a)=-1$ according to whether $a$ is a square of $\F_q$ or not, respectively. 

It is well known that if $q$ is odd, $Q_B$ can assume a diagonal form and any diagonalization has the form $\sum_{i=1}^na_iw_i^2, a_i\in \F_q^*$ with $\prod_{i=1}^n\chi_q(a_i)\in \{\pm1\}$ being a constant that depends only on $B$. We denote such constant by $\chi_q(B)$. For more details on this fact, see Subsection 2.1 in~\cite{shparlinski}. The following lemma provides some facts about bilinear forms.

\begin{lemma}\label{lem:bilinear}
Let $B, B_0$ be two non degenerate symmetric $\F_q$-bilinear forms of $\F_{q^n}$ and let $U\subseteq \F_{q^n}$ be an $\F_q$-vector space. Then the following hold: 
\begin{enumerate}[1.]
%\item $U$ admits  $B$-orthogonal $\F_q$-basis, i.e., an $\F_q$-basis $\mathcal B$ such that $B(a, b)=0$ for every $a, b\in \mathcal B$ with $a\ne b$; 
\item The $\F_q$-vector space $U_0$ comprising the elements $w$ such that $B(u, w)=0$ for every $u\in U$ has dimension $n-d_U$;
\item If $q$ is odd, we have that $\chi_q(B)=\chi_q(B_0)$ if and only if there exists an invertible $\F_q$-linear map $T:\F_{q^n}\to \F_{q^n}$ such that $B(x, y)=B_0(T(x), T(y))$ for every $x, y\in \F_{q^n}$.
\end{enumerate}
\end{lemma}
\begin{proof}
Item 1 follows by the fact that $U_0$ is precisely the orthogonal complement of $U$ (with respect to $B$, which is non degenerate). For the proof of item 2,  see page 79 in~\cite{book}.

\end{proof}

The following technical lemma provides information on the largest cardinality $M_B$ of a pairwise $B$-orthogonal set in $\F_{q^n}$ and it is crucial in the proof of Theorem~\ref{thm:n-1}. The number $M_B$ is known for fields of odd order~\cite{ah} (see also~\cite{shparlinski} for similar questions). Here we present an extended version of such result, providing partial information for fields of even order.
\begin{lemma}\label{thm:crucial}
Let $B$ be a non degenerate symmetric $\F_q$-bilinear form of $\F_{q^n}$ and let $t(B)$ be the largest dimension of an $\F_q$-vector space $W\subseteq \F_{q^n}$ such that $B(u, w)=0$ for every $u, w\in W$. Moreover, let $M_B$ be the largest cardinality of a set $E\subseteq \F_{q^n}$ such that $B(u, w)=0$ for every $u, w\in E$ with $u\ne w$. Then the following hold:
\begin{enumerate}[1.]
\item Suppose that $q$ is odd. Then $M_B=q^{t(B)}+n-2t(B)$, where $t(B)=\frac{n-1}{2}$ if $n$ is odd and 
$$t(B)=\frac{n+\chi_q(B)\cdot \chi_q(-1)^{n/2}-1}{2}\in \left\{\frac{n}{2}, \frac{n}{2}-1\right\},$$ if  $n$ is even.
\item If $q=2$ and $t(B)\le 2$, then $M_B\le n+1$.
\item If $q$ is even and either $q>2$ or $t(B)\ge 3$, then $M_B\le q^{t(B)}+n-2t(B)$ and $t(B)\le n/2$.
\end{enumerate}
\end{lemma}
\begin{proof}
Let $E\subseteq \F_{q^n}$ with $\# E=M_B$ be such that $B(u, w)=0$ for every $u, w\in E$ with $u\ne w$. We clearly have $0\in E$. Item 1 follows directly by Theorem 4 in~\cite{ah}, where the value of $M_B-1$ is computed (this is because they impose the restriction $0\not\in E$). For the sake of completeness, we provide an important remark: in the notation of Theorem 4 in~\cite{ah}, if $n$ is even, we have that $\epsilon(B)=1$ if and only if $\chi_q(B)\cdot \chi_q(-1)^{n/2}=1$. 

We now assume that $q$ is even. Lemma~3 in~\cite{ah} entails that $E=U\cup S$, where $U$ is an $\F_q$-vector space of dimension $t\le t(B)$ satisfying $B(u, w)=0$ for every $u, w\in U$, and  $\# S\le n-2t$. In particular, $M_B\le q^{t}+n-2t$ for some $0\le t\le t(B)$. If $q=2$ and $t(B)\le 2$, we easily verify that $M_B\le n+1$. This proves item 2. We proceed to the proof of item 3. Suppose that $q$ is even and either $q>2$ or $t(B)\ge 3$. If $f(s)= q^s+n-2s$, by computing the derivative of $f$, we directly verify that $f(s)$ is increasing on $s\in \Z_{\ge a}$, where $a=0$ if $q>2$ and $a=3$ if $q=2$. Moreover, for $q=2$, we have that $f(3)>\max\{f(0), f(1), f(2)\}$. Since either $q>2$ or $t(B)\ge 3$, we conclude that $M_B\le f(t(B))=q^{t(B)}+n-2t(B)$. The inequality $t(B)\le n/2$ follows directly by the definition of $t(B)$ and item 1 in Lemma~\ref{lem:bilinear}.
\end{proof}

\subsection{The trace map}
We now introduce the trace function of finite fields.

\begin{definition}
Let $q$ be a prime power and let $n\ge 2$ be a positive integer. The trace of $\F_{q^n}$ over $\F_q$ is the map $a\mapsto \Tr_n(a):= \sum_{i=0}^{n-1}a^{q^i}$. 
\end{definition}

The following lemma displays some basic properties of the trace functions. Its proof is quite elementary so we skip details.

\begin{lemma}\label{lem:trace}
Let $q$ be a prime power and let $n\ge 2$.
\begin{enumerate}[1.]
\item The map $a\mapsto \Tr_n(a)$ is $\F_q$-linear;
\item The image set of $\Tr_n$ over $\F_{q^n}$ equals $\F_q$. In particular, for every $a\ne 0$, the map $(x, y)\mapsto \Tr_n(axy)$ is a non degenerate symmetric $\F_q$-bilinear form of $\F_{q^n}$. 
\end{enumerate}
\end{lemma}

When $q$ is odd and $\chi_q(B)=\chi_q(B_0)$, we usually say that $B$ and $B_0$ are equivalent bilinear forms. We end this section showing that there are exactly two non equivalent bilinear forms arising from trace functions over finite fields of odd characteristic.

\begin{lemma}\label{lem:trace-equiv}
Let $q$ be an odd prime power and, for $\lambda\in \F_{q^n}^*$, set $B_{\lambda}(x, y)= \Tr_n(\lambda xy)$. Then $\chi_q(B_{\lambda})=\chi_q(B_1)$ if and only if $\lambda$ is a square.
\end{lemma}
\begin{proof}
Suppose that $\lambda=a^2\ne 0$, hence $\Tr_n(\lambda xy)=\Tr_n(T(x)T(y))$ for $T(x)=ax$ and so item 2 in Lemma~\ref{lem:bilinear} entails that $\chi_q(B_{\lambda})=\chi_q(B_1)$. Conversely, suppose by contradiction that $\lambda$ is not a square and $\chi_q(B_{\lambda})=\chi_q(B_1)$. Item 2 in Lemma~\ref{lem:bilinear} entails that there exists an invertible $\F_q$-linear linear map $T:\F_{q^n}\to \F_{q^n}$ such that $\Tr_n(\lambda xy)=\Tr_n(T(x)T(y))$ for every $x, y\in \F_{q^n}$. For each $a\in \F_{q^n}^*$, let $V_{a}$ be the $\F_q$-vector space comprising the elements $x\in \F_{q^n}$ such that $T(x)=ax$. Suppose that $V_a$ has dimension $t>0$. Hence 
$\Tr_n(x(\lambda y- aT(y)))=0$ for every $y\in \F_{q^n}$ and every $x\in V_a$. In particular, if $W$ denotes the image set of the $\F_q$-linear map $L_a:y\mapsto \lambda y -a T(y)$, we have that $B_1(x, w)=0$ for every $x\in V_a, w\in W$. Item 1 in Lemma~\ref{lem:bilinear} implies that $W$ has dimension at most $n-t$. From the Rank-Nullity Theorem, we have that $\ker L_a=V_{\frac{\lambda}{a}}$ has dimension $s\ge n-(n-t)=t$. By the same reasoning, $V_{a}=V_{\frac{\lambda}{\frac{\lambda}{a}}}$ has dimension $t\ge s$ and so $s=t$. In particular, the possible values of $T(x)/x$ come in pairs $(a, \lambda/a)$ and the number of non zero solutions to the equation $T(x)=e x$ is the same for $e=a, \lambda/a$. Since $\lambda$ is not a square, we have that $a\ne \lambda/a$ for every $a\in \F_{q^n}^*$. Therefore, since $T(x)$ is invertible and $T(0)=0$, we have that 
$0\ne \prod_{x\in \F_{q^n}^*}T(x)=\prod_{x\in \F_{q^{n}}^*}x$ and so 
$$1=\prod_{x\in \F_{q^n}^*}\frac{T(x)}{x}=\lambda^{\frac{q^n-1}{2}}.$$
The latter implies that $\lambda$ is a square, a contradiction with our initial assumption.
\end{proof}

\section{The case $d_U=n-1$}

In this section we consider the $\F_{q}$-vector spaces $U \subsetneq \F_{q^n}$ of dimension $n-1$. Let $U_{q, n}:=\{\alpha\in \F_{q^n}\,|\, \Tr_n(\alpha)=0\}=\ker \Tr_n$ be the $\F_q$-vector space comprising the elements of $\F_{q^n}$ with trace zero over $\F_q$. From item 2 in Lemma~\ref{lem:trace} and the Rank-Nullity Theorem, $U_{q, n}$ has dimension $n-1$. 

It is direct to verify that the number of $\F_q$-vector spaces in $\F_{q^n}$ of dimension $n-1$ equals $\frac{q^n-1}{q-1}$. Moreover, if $\delta\in \fqn^*$ is such that $\delta U_{q, n}=U_{q, n}$, then the $q^{n-1}$-degree polynomials $T_n(x)=\sum_{i=0}^{n-1}x^{q^i}$ and $T_n(\delta^{-1}x)=\sum_{i=0}^{n-1}\delta^{-q^i}x^{q^i}$ vanish at every element of $U_{q, n}$. The latter implies that such polynomials are equal up to a multiplication by a scalar and a simple calculation entails that $\delta\in \F_q$. Conversely, for $\delta\in \F_q^*$, we have that $\delta U_{q, n}=U_{q, n}$ since $U_{q, n}$ is an $\F_q$-vector space. In particular, as $\delta$ runs over $\F_{q^n}^*$, the set $\delta U_{q, n}$ runs over $\frac{q^n-1}{q-1}$ distinct $\F_q$-vector spaces of dimension $n-1$. Therefore, for every $\F_{q}$-vector space $U \subsetneq \F_{q^n}$ of dimension $n-1$, there exists $\delta\in \F_{q^n}^*$ such that $U=\delta U_{q, n}$. 

In the following lemma we prove that there exist at most two classes of non isomorphic graphs arising from the graphs $G_U$ with $d_U=n-1$.

\begin{lemma}\label{lem:equiv}
Let $U\subsetneq \F_{q^n}$ be an $\F_q$-vector space of dimension $n-1$. Then the following hold:
\begin{enumerate}[1.]
\item If $q$ is even or $qn$ is odd, there exists $a\ne 0$ such that $U=a^2U_{q, n}$. In this case, the graphs $G_U$ and $G_{U_{q, n}}$ are isomorphic.
\item If $q$ is odd and $n$ is even, there exists a non square $\lambda$ (not depending on $U$) such that $U=a^2 \cdot \varepsilon U_{q, n}$ with $a \ne 0$ and $\varepsilon\in \{1, \lambda\}$. In this case, the graphs $G_U$ and $ G_{\varepsilon U_{q, n}}$ are isomorphic.
\end{enumerate}

\end{lemma}

\begin{proof}
We have seen that there exists $\delta\in \F_{q^n}^*$ such that $U=\delta U_{q, n}$. We consider the two items separately. 

\begin{enumerate}[1.]
\item It suffices to prove that we can take $\delta$ as a square; this is trivially true if $q$ is even. Suppose that $q, n$ are odd and $\delta$ is not a square. In particular, $N=\frac{q^n-1}{q-1}$ is odd and so $\theta^{(q^n-1)/2}=(\theta^{(q-1)/2})^N=(-1)^N-1$, where $\theta$ is any non square of $\F_q^*$. In particular, $\theta$ is not a square and so $\delta'=\theta \delta$ is a square.
Since $\theta\in \F_q$ and $U_{q, n}$ is an $\F_q$-vector space, we have that $\delta U_{q, n}=\delta' U_{q, n}$. 
\item It suffices to consider the case where $\delta$ is not a square. In this case, $\frac{\delta}{\lambda}=a^2$ for some $a\ne 0$ and so $U=\lambda a^2 U_{q, n}$. 
\end{enumerate}
The isomorphism statement in both items follows directly by the fact that $(a_1, a_2)\in E(G_U)$ if and only if $(aa_1, aa_2)\in E(G_{a^2U})$.
\end{proof}

Motivated by the previous lemma, we introduce the following definition.

\begin{definition}
Let $q$ be an odd prime power and let $n\ge 2$ be even. If $U\subsetneq \F_{q^n}$ is an $\F_q$-vector space of dimension $n-1$, we set $s(U)=1$ if $U=a^2U_{q, n}$ for some $a\ne 0$ and, otherwise, $s(U)=-1$.
\end{definition}

\begin{remark}\label{rem:cool}
We have seen that the number of $\F_q$-vector spaces in $\F_{q^n}$ of dimension $n-1$ equals $N=\frac{q^n-1}{q-1}$. Moreover, if $q$ is odd and $n$ is even, $N/2$ is an integer and then, for every $b\in \F_q^*$, we have that $b^{(q^n-1)/2}=b^{\frac{N}{2}(q-1)}=1$. In particular, every element of $\F_q^*$ is a square. The latter implies that we have at most $N/2$ distinct vector spaces $U$ of the form $a^2U_{q, n}$, i.e., with $s(U)=1$. A similar argument implies the same for $s(U)=-1$. In conclusion, the number of $U$'s with $s(U)=1$ and with $s(U)=-1$ coincide and are equal to $N/2$. 
\end{remark}

The next lemma ensures the existence of a special $\F_q$-basis for finite field extensions. For its proof, see Theorems 1 and 2 in~\cite{self-dual}.
\begin{lemma}\label{lem:basis}
Let $q$ be a prime power and $n\ge 2$ be a positive integer. Then there exists an $\F_q$-basis $\mathcal B_{q, n}=\{\beta_1, \ldots, \beta_n\}$ for $\F_{q^n}$ such that 
$$\Tr_n(\beta_i\beta_j)=\begin{cases}0 \,\,\text{if}\,\, i\ne j, \\ 1 \,\,\text{if}\,\, i=j>1,\\
\mu\,\,\text{if}\,\, i=j=1,\end{cases}$$
 for some $\mu \in \F_q^*$. Moreover, there exists such a basis with $\mu=1$ if and only if $q$ is even or $q\cdot n$ is odd. When $q$ is odd and $n$ is even, we can only take $\mu$ as an arbitrary non square of $\F_q$.
\end{lemma}

We obtain the following result.

\begin{theorem}\label{thm:n-1}
Let $q$ be a prime power and let $n\ge 2$ be an integer such that either $q>2$ or $n\ge 6$. If $U\subsetneq \F_{q^n}$ is an $\F_q$-vector space of dimension $n-1$, then the following hold:

\begin{enumerate}[1.]
\item if $q$ is even or $n$ is odd, $\omega(G_U)=q^{\lfloor n/2\rfloor}+n-2\lfloor n/2\rfloor$;
\item if $q$ is odd and $n$ is even, then $\omega(G_U)=q^{n/2}$  if 
$$(q\mod 4, n\mod 4, s(U))\in \{(\pm 1, 0, -1), (-1, 2, 1), (1, 2, -1)\},$$
and $\omega(G_U)=q^{n/2-1}+2$ if
$$(q\mod 4, n\mod 4, s(U))\in \{(\pm 1, 0, 1), (-1, 2, -1), (1, 2, 1)\}.$$
\end{enumerate}
Moreover, if $q=2$ and $2\le n\le 5$, then $\omega(G_U)=n+1$.
\end{theorem}

\begin{proof}
Let $\mathcal B_{q, n}=\{\beta_1, \ldots, \beta_n\}$ be an $\F_q$-basis for $\F_{q^n}$ as in Lemma~\ref{lem:basis}. We prove items 1 and 2 separately.

\begin{enumerate}[1.]
\item Suppose that $q$ is even or $n$ is odd. From Lemmas~\ref{lem:equiv} and~\ref{lem:basis}, we can assume that $U=U_{q, n}$ and $\Tr_{n}(\beta_i^2)=1$ for every $1\le i\le n$. In particular, for every $x\ne y\in \F_{q^n}$ with $x=\sum_{i=1}^{n}x_i \beta_i, y=\sum_{i=1}^{n}y_i \beta_i$, we have that $(x, y)\in E(G_U)$ if and only if
$$0=\Tr_n(xy)=\sum_{i=1}^nx_iy_i=: B(x, y).$$
Lemma~\ref{lem:trace} entails that $B$ is a non degenerate symmetric $\F_q$-bilinear form of $\F_{q^n}$. In particular, under the notation of Lemma~\ref{thm:crucial}, $\omega(G_U)=M_{B}$. The result for $q$ odd now follows directly by item 1 in Lemma~\ref{thm:crucial}. 

Assume now that $q$ is even and $n=2n_0+R$ with $R\in\{0, 1\}$.  If $W$ is the $\F_q$-vector space generated by the elements $\beta_{2i-1}+\beta_{2i}, i=1, \ldots, n_0$, then $W$ has dimension $n_0$ and $B(u, w)=0$ for every $u, w\in W$. From item 3 in Lemma~\ref{thm:crucial}, we have that $t(B)=n_0$. Moreover, under our initial assumptions, either $q>2$ or $n_0\ge 3$. In particular, from item 3 in Lemma~\ref{thm:crucial}, we conclude that $M_B\le q^{n_0}+n-2n_0=q^{n_0}+R$. Therefore, it suffices to produce a set $E$ such that $\# E=q^{n_0}+R$ and $B(u, w)=0$ for every $u, w\in E$ with $ u\ne w$. We may take $E=W$ if $n$ is even and $E=W\cup\{\beta_{2n_0+1}\}$ if $n$ is odd.

\item Suppose that $q$ is odd and $n$ is even. From Lemma~\ref{lem:equiv}, we may assume that $U=\varepsilon U_{q, n}$ with $\varepsilon\in \{1, \lambda\}$, where $\lambda\in \F_{q^n}^*$ is a fixed non square (not depending on $U$). In particular, for $x\ne y\in \F_{q^n}$, we have that $(x, y)\in E(G_U)$ if and only if 
$$0=\Tr_n(\varepsilon^{-1} xy)=: B_{\varepsilon}(x, y).$$
Lemma~\ref{lem:trace} entails that $B_{\varepsilon}$ is a non degenerate symmetric $\F_q$-bilinear form of $\F_{q^n}$. In particular, from Lemma~\ref{thm:crucial}, $\omega(G_U)=M_{B_{\varepsilon}}$ and then by the same lemma we just need to compute the value of $\chi_q(B_{\varepsilon})$.
We observe that if $x=\sum_{i=1}^nx_i\beta_i\in \F_{q^n}$, then
$B_1(x, x)=\mu x_1^2+\sum_{i=2}^nx_i^2$ and so $\chi_q(B_1)=\chi_q(\mu)=-1$, where $\mu$ is as in Lemma~\ref{lem:basis} (recall that we are under the hypothesis that $q$ is odd and $n$ is even). Since $\lambda$ is not a square, Lemma~\ref{lem:trace-equiv} entails that $\chi_q(B_{\lambda})\ne \chi_q(B_{1})$ and so $\chi_q(B_{\lambda})=1$.  In particular, $\chi_q(B_{\varepsilon})=-s(U)$ and the result follows by Lemma~\ref{thm:crucial}, after computing the values of $s(U)\cdot \chi_q(-1)^{n/2}\in\{ \pm 1\}$.
\end{enumerate}

Assume that $q=2$ and $2\le n\le 5$. Lemma~\ref{lem:trace} entails that $\omega(G_U)\le n+1$. Following the notation and arguments given in item 1 for $q$ even, just take $E=\{\beta_1, \ldots, \beta_n, 0\}$.
\end{proof}
We end this paper with an interesting remark.
\begin{remark}
Theorem~\ref{thm:n-1} entails that if $q$ is odd, $n\ge 2$ is even and $d_U=n-1$, then $\omega(G_U)$ may assume two distinct values, according to whether $U=a^2U_{q, n}$ for some $a\ne 0$ or not. In particular, from Lemma~\ref{lem:equiv}, we conclude that there are precisely two isomorphism classes among the graphs $G_U$ with $d_U=n-1$. Moreover, from Remark~\ref{rem:cool}, each class contains exactly $\frac{q^n-1}{2(q-1)}$ graphs. A similar argument (extending Lemma~\ref{lem:equiv} and using Proposition~\ref{prop:basic}) implies the same result for $\F_q$-vector spaces of dimension $1$ if $q>3$ is odd and $n$ is even.
\end{remark}

\section*{Acknowledgements}
We thank the anonymous reviewers for their helpful comments and suggestions. This work was partially supported by CNPq (309844/2021-5).

\end{document}